\documentclass[12pt]{amsart}
\usepackage{amsmath, amssymb}
\let\temp\rmdefault
\usepackage{mathpazo}
\usepackage{lmodern}
\usepackage{mathrsfs}
\let\rmdefault\temp
\usepackage[margin=1in]{geometry}
\usepackage[amsmath]{empheq}
\usepackage{pdfrender,xcolor}
\usepackage{comment}
\usepackage{mathtools}
\usepackage{multicol}
\usepackage{enumitem, multicol}
\usepackage{setspace}
\usepackage[most]{tcolorbox}
\tcbset{colback=yellow!10!white, colframe=blue!50!blue, highlight math style= {enhanced, 
		colframe=blue,colback=red!10!white,boxsep=0pt}}
\usepackage[linkcolor=blue]{hyperref}
\hypersetup{colorlinks = true}
\newtheorem{theorem}{Theorem}[section]
\newtheorem{lemma}[theorem]{Lemma}

\newtheorem{proposition}[theorem]{Proposition}

\newtheorem{question}{Question}

\newtheorem{definition}[theorem]{Definition}
\newtheorem{example}[theorem]{Example}

\newtheorem{remark}[theorem]{Remark}

\theoremstyle{definition}
\numberwithin{equation}{section}

\newcommand{\norm}[1]{\|#1\|}

\begin{document}
	
	\title[On the geometry of $G$-norm]{On the geometry of $G$-norm}
	\author{Lakshmi Kanta Dey}
	\address{Lakshmi Kanta Dey, Department of Mathematics, National Institute of Technology, Durgapur-713209, West Bengal, India.}
	\email{lakshmikdey@yahoo.co.in}
	\author{Subhadip Pal}
	\address{Subhadip Pal, Department of Mathematics, National Institute of Technology, Durgapur-713209, West Bengal, India.}
	\email{palsubhadip2@gmail.com}
	\keywords{Relative spear operator; Numerical index; Smoothness; Isometry.\\
		The research of Mr. Subhadip Pal is supported by the University Grants Commission, Government of India.}

	\begin{abstract}
		Let $X$ and $Y$ be Banach spaces and let $G \in L(X,Y)$ with $\|G\|=1$. We study the geometry of $G$-(semi-)norm on $L(X,Y)$, defined by
		\[
		\|T\|_G := \inf_{\delta>0}\sup\{\|Tx\|: \|x\|=1, \|Gx\|>1-\delta\},
		\]
		considering it as a norm ($G$-norm), and further explore the associated numerical indices. In particular, we characterize relative spear operators, that is, operators for which the numerical radius with respect to $G$ coincides with the $G$-norm. Relations among the numerical indices and their invariance under isometric isomorphisms are established. We further obtain a description of the dual unit ball of $(L(X,Y),\|\cdot\|_G)$ and characterize smooth points of its unit ball. In finite-dimensional Hilbert spaces, we prove that relative spear operators are partial isometries. Finally, we establish some equivalent criteria for which the $G$-norm is achieved by the norm attainment set of a norm-attaining operator $G$.
	\end{abstract}   
	
	\maketitle
	\section{Introduction}
	Throughout the paper, $X$, $Y$ and $Z$ denote the Banach spaces over the scalar field $\mathbb{K}$, where $\mathbb{K}=\mathbb{R}$ or $\mathbb{C}$. The zero vector of a Banach space is denoted by $\theta$, except the scalar field. The unit sphere and closed unit ball of $X$ are denoted by $S_X$ and $B_X$, respectively. The dual space of $X$ is denoted by $X^*$. The space of all bounded linear operators from $X$ to $Y$ is denoted by $L(X,Y)$ and we write $L(X)$ when $X=Y$. We denote by $S^1:=\{\lambda\in\mathbb{K} : |\lambda|=1\}.$ We now introduce the relevant definitions. For $x\in X$, we denote by $J_{\|.\|}(x) :=\{f\in S_{X^*}: f(x)=\|x\|\}$ the set of \emph{support functionals at $x$}. $M_G :=\{x\in S_X : \|Gx\|=\|G\|\}$ denotes the set of norm-attaining points of $G\in L(X, Y)$. The classical numerical range of an operator $T\in L(X)$ is defined by
    \[
    V(T):=\{x^*(Tx): x\in S_X,\ x^*\in S_{X^*},\ x^*(x)=1\},
    \]
    and the associated numerical radius and numerical index are given by
    \[
    \nu(T):=\sup\{|\lambda|:\lambda\in V(T)\}\qquad \text{and} \qquad n(X):=\inf\{\nu(T): T\in L(X),\ \|T\|=1\}.
    \]
    A fundamental line of research in operator theory studies the relationship between the operator norm and the numerical radius \cite{BonsallDuncan, Duncan}. In particular, the condition $n(X)=1$, that is, $\nu(T)=\|T\|$ for all $T\in L(X)$, is known to be equivalent to the identity
    \begin{equation}\label{eq:intro-classical}
    \max_{\omega\in S^1}\|{\rm Id}+\omega T\|=1+\|T\|.
    \end{equation}
    Motivated by this framework, Ardalani \cite{Ardalani} introduced a generalized \emph{numerical range and radius with respect to a fixed operator $G\in L(X,Y)$}, defined by
    \begin{align*}
        V_G(T):=\bigcap_{\delta>0}\overline{\{y^*(Tx): x\in S_X,\ y^*\in S_{Y^*},\ \Re y^*(Gx)>1-\delta\}}\\
        \text{and}\quad\nu_G(T):=\max\{|\lambda|:\lambda\in V_G(T)\}\quad\text{respectively}.
    \end{align*}
    It is worth mentioning that the above expression can be treated as
    \begin{align*}
		\nu_G(T) & := \sup\{|\phi(T)| : \phi \in L(X,Y)^*, \|\phi\| = \phi(G) = 1\}\\
		& = \inf_{\delta>0} \sup\{|y^*(Tx)| : y^* \in S_{Y^*}, x \in S_X, \Re  y^*(Gx) > 1 - \delta\}\quad \text{(see, \cite[Theorem 2.1]{MartinRange})}.    
	\end{align*}
    This notion extends the classical numerical radius and leads naturally to the concept of spear operators: a norm-one operator $G$ is called a \emph{spear operator} if for every $T\in L(X,Y)$ there exists $\omega\in S^1$ such that
    \[
    \|G+\omega T\|=1+\|T\|.
    \]
    In particular, the identity operator is a spear operator if and only if $n(X)=1$. More recently, Kadets et. al. \cite{Generatingoperator} introduced a new definition of (semi-)norm in $L(X, Y)$:
	\begin{definition}\cite{Generatingoperator}
		Let $X, Y$ and $Z$ be Banach spaces and let $G \in L(X,Y)$ be a norm-one operator. For $T \in L(X,Z)$, we define the (semi-)norm of $T$ relative to $G$ by
		\[
		\|T\|_G := \inf_{\delta>0}\sup\{\|Tx\|:x\in S_X, \|Gx\|>1-\delta\}.
		\]
	\end{definition}
	\noindent
    This construction can be viewed as a natural generalization of the usual operator norm, which is recovered when $G$ is the identity. When $Z = Y$, one has the inequalities
    \[
    \nu_G(T)\leq \|T\|_G\leq \|T\|, \qquad T\in L(X,Y),
    \]
    and the associated numerical indices are defined by
	\begin{align*}
		n_{G}(X,Y)
		&:= \inf \bigl\{ \nu_{G}(T) : \|T\| = 1 \bigr\}= \max \bigl\{ k \ge 0 : k \|T\| \le \nu_{G}(T),\ T \in  L(X, Y) \bigr\},\\
		n^{(1)}_{G}(X,Y)
		&:= \inf \bigl\{ \|T\|_G : \|T\| = 1 \bigr\}= \max \bigl\{ k \ge 0 : k \|T\| \le \|T\|_G,\ T \in  L(X, Y) \bigr\},\\
		n^{(2)}_{G}(X,Y)
		&:= \inf \bigl\{ \nu_{G}(T) : \|T\|_{G} = 1 \bigr\}= \max \bigl\{ k \ge 0 : k \|T\|_{G} \le \nu_{G}(T),\ T \in  L(X, Y) \bigr\}.
	\end{align*}
	By introducing the definition of $\|T\|_G$, Kadets et. al. studied about the operators $T$ which satisfy the relation $\|T\|_G=\|T\|$, see, \cite{Generatingoperator}. Throughout this article, we assume that $\|\cdot\|_G$ defines a norm and refer to it as the $G$-norm.

	\medskip

	In the spirit of the above developments, we introduce the following definition:
	\begin{definition}
		An operator $G\in L(X,Y)$ with $\|G\|=1$ is called a relative spear operator if $\nu_G(T)=\|T\|_G$ for all $T\in L(X,Y),$ i.e., $n^{(2)}_{G}(X,Y)=1$.
	\end{definition}
    \noindent
    This concept provides a connection between the geometry induced by the $G$-norm and the numerical radius with respect to $G$.

    \medskip

    The aim of this paper is to investigate the geometry of the $G$-norm and its interaction with the numerical radius $\nu_G$, with particular emphasis on relative spear operators. Our main contributions can be summarized as follows.

    \medskip

    \noindent
    First, we obtain a characterization of relative spear operators analogous to \eqref{eq:intro-classical}, showing that the equality $\nu_G(T)=\|T\|_G$ is equivalent to the identity
    \[
    \max_{\omega\in S^1}\|G+\omega T\|_G=1+\|T\|_G.
    \]
    We also establish a relationship among the numerical indices $n_G$, $n_G^{(1)}$ and $n_G^{(2)}$, and prove that $n_G^{(2)}$ is invariant under isometric isomorphisms. Next, we study the geometric structure of the Banach space $(L(X,Y),\|\cdot\|_G)$. In particular, we provide a description of its dual unit ball and use this representation to characterize smooth points of the unit ball. Finally, we specialize our analysis to Hilbert spaces. In the finite-dimensional setting, we show that every relative spear operator is a partial isometry. We further obtain conditions under which the $G$-norm can be computed using the norm-attainment set $M_G$. We conclude the article with three questions for the readers concerning possible extensions of these results to the infinite-dimensional and general Banach space settings.

	\section{Preliminaries}
	We begin this section with some preliminary results that will be needed in the sequel. We first establish an alternative expression of $\|.\|_G$ in $L(X, Y)$. Before that for $T\in L(X, Y)$ we define 
	\[
	S_G(T):=\bigcap_{\delta>0}\overline{\{y^*(Tx): x\in S_X,\ y^*\in S_{Y^*},\ \|Gx\|>1-\delta\}}
	\]
	and 
	\[
	\widetilde{S}_G(T):=\bigcap_{\delta>0}\overline{\{\|Tx\|: x\in S_X, y^*\in S_{Y^*}, \|Gx\|>1-\delta\}}.
	\]
	\begin{proposition}\label{DP1}
		Let $G\in L(X,Y)$ with $\|G\|=1$. Then 
		\[
		\|T\|_G=\max\{|\lambda|:\lambda\in S_G(T)\}=\max\{\lambda:\lambda\in \widetilde{S}_G(T)\}.
		\]
	\end{proposition}
	\begin{proof}
		For $T\in L(X,Y)$ and $\delta>0$ define $E_{\delta}(T):=\bigl\{y^{*}(Tx): x\in S_X, y^{*}\in S_{Y^{*}}, \|Gx\|>1-\delta \bigr\}$ and $\widetilde{E}_{\delta}(T):=\bigl\{\|Tx\|: x\in S_X, \|Gx\|>1-\delta \bigr\}.$ Then $S_G(T)=\bigcap_{\delta>0}\overline{E_{\delta}(T)}$ and $\widetilde{S}_G(T)=\bigcap_{\delta>0}\overline{\widetilde{E}_{\delta}(T)}$. By using the Hahn-Banach Theorem, for every fixed $\delta>0$ we have
		\begin{align*}
			\sup\{\|Tx\|:x\in S_X, \|Gx\|>1-\delta\}&= \sup\{|y^{*}(Tx)|: x\in S_X, y^{*}\in S_{Y^{*}}, \|Gx\|>1-\delta \}\\
			&= \sup\{|\lambda|:\lambda\in E_{\delta}(T)\}.    
		\end{align*}
		Therefore, by the definition of the $G$-norm,
		\[
		\|T\|_{G}=\inf_{\delta>0}\sup\{\|Tx\|:x\in S_X,\|Gx\|>1-\delta\}=\inf_{\delta>0} \sup\{|\lambda|:\lambda\in E_{\delta}(T)\}.
		\]
		Since the family $\{\overline{E_{\delta}(T)}\}_{\delta>0}$ is a decreasing family of nonempty compact subsets of $\mathbb{C}$ as $\delta$ tends to zero, we have 
		\[
		\inf_{\delta>0}\max\{|\lambda|:\lambda\in\overline{E_{\delta}(T)}\}=\max\{|\lambda|:\lambda\in \bigcap_{\delta>0}\overline{E_{\delta}(T)}\}.
		\]
		Also, we have $\inf_{\delta>0}\max\{|\lambda|:\lambda\in\overline{E_{\delta}(T)}\}=\inf_{\delta>0} \sup\{|\lambda|:\lambda\in E_{\delta}(T)\}$. Applying this we get $\|T\|_{G}=\max\{|\lambda|:\lambda\in\bigcap_{\delta>0}\overline{E_{\delta}(T)}\}=\max\{|\lambda|:\lambda\in S_G(T)\}.$ A similar approach gives us $\|T\|_{G}=\max\{\lambda:\lambda\in\bigcap_{\delta>0}\overline{\widetilde{E}_{\delta}(T)}\}=\max\{\lambda:\lambda\in \widetilde{S}_G(T)\}.$ This completes the proof.
	\end{proof}
	
	We further reduce the expression of $S_G(T)$ in the following result.
	
	\begin{proposition}\label{P2.2}
		Let $G\in S_{L(X,Y)}$ and let $T\in L(X,Y)$. Then
		\begin{itemize}
			\item[(i)] $S_G(T)=
			\Big\{\lim_{n} y_n^*(Tx_n):
			x_n\in S_X, y_n^*\in S_{Y^*}, \|Gx_n\|\to 1 \Big\}.$
			\item[(ii)] $\widetilde{S}_G(T)=
			\Big\{\lim_{n} \|Tx_n\|:
			x_n\in S_X, \|Gx_n\|\to 1 \Big\}.$ 
		\end{itemize}
	\end{proposition}
	
	\begin{proof}
		We only prove $(i)$ and the proof of $(ii)$ can be easily done in a similar way. Let $A:=\Big\{\lim_{n} y_n^*(Tx_n):x_n\in S_X,\ y_n^*\in S_{Y^*},\ \|Gx_n\|\to 1 \Big\}.$ We first prove that $A\subset S_G(T)$. Let $t\in A$. Then there exist sequences $x_n\in S_X$ and 
		$y_n^*\in S_{Y^*}$ such that
		\[
		\|Gx_n\|\to 1\quad\text{and}\quad y_n^*(Tx_n)\to t.
		\]
		Fix $\delta>0$. Since $\|Gx_n\|\to 1$, there exists $N$ such that $\|Gx_n\|>1-\delta$ for all $n\ge N.$ Hence for $n\ge N$,
		\[
		y_n^*(Tx_n)\in \{y^*(Tx):y^*\in S_{Y^*},\,x\in S_X,\,\|Gx\|>1-\delta\}.
		\]
		Since $y_n^*(Tx_n)\to t$, it follows that $t\in \overline{\{y^*(Tx):y^*\in S_{Y^*},\,x\in S_X,\,\|Gx\|>1-\delta\}}.$ Because this holds for every $\delta>0$, we conclude that $t\in S_G(T)$ and consequently, $A\subset S_G(T)$. We next prove that $S_G(T)\subset A$. Let $t\in S_G(T)$. Then for every $\delta>0$,
		\[
		t\in
		\overline{\{y^*(Tx):y^*\in S_{Y^*},\,x\in S_X,\,\|Gx\|>1-\delta\}}.
		\]
		Take $\delta_n=1/n$. For each $n$ there exist $x_n\in S_X$ and
		$y_n^*\in S_{Y^*}$ such that $\|Gx_n\|>1-\frac1n$ and $|y_n^*(Tx_n)-t|<\frac1n.$ Therefore, $y_n^*(Tx_n)\to t$ and $\|Gx_n\|\to 1.$ Thus, $t\in A$ and consequently, $S_G(T)\subset A$. Since both inclusions hold, we obtain
		\[
		S_G(T)=
		\Big\{\lim_{n} y_n^*(Tx_n):
		x_n\in S_X,\ y_n^*\in S_{Y^*},\ \|Gx_n\|\to 1 \Big\}.
		\]
	\end{proof}

	\begin{remark}\label{R2.4}
		Analogous to Proposition \ref{P2.2}, one can obtain an alternative description of $V_G(T)$ as following:
		\[
		V_G(T)=\Big\{\lim_{n} y_n^*(Tx_n):x_n\in S_X, y_n^*\in S_{Y^*}, \Re y_n^*(Gx_n)\to 1 \Big\}.
		\]
        It is worth mentioning that various expressions of $V_G(T)$ can be seen in \cite{MMQRS}.
	\end{remark}
	
	It is easy to see that in finite-dimensional setting, the expressions of $S_G(T)$ and $V_G(T)$ obtained in Proposition \ref{P2.2} and Remark \ref{R2.4} respectively, reduce significantly, we omit the proof as it is straightforward:

	\begin{proposition}\label{P2.5}
		Let $X, Y$ be finite-dimensional Banach spaces. Assume that $G\in S_{L(X,Y)}$ and let $T\in L(X,Y)$. Then
		\begin{itemize}
			\item[(i)] $S_G(T)=\Big\{ y^*(Tx):x\in S_X, y^*\in S_{Y^*},\ \|Gx\|= 1 \Big\}.$
			\item[(ii)] $V_G(T)=\Big\{y^*(Tx):x\in S_X, y^*\in S_{Y^*}, y^*(Gx)= 1 \Big\}.$
		\end{itemize}
	\end{proposition}

    We end this section by presenting the Bipolar Theorem which will be used in the discussion of the geometry of $(L(X, Y), \|.\|_G)^*$. For $A \subset X$ and a subspace $\mathcal{X} \subset X^*$, we define the \emph{polar} of $A$ by $A^\circ = \{ f \in \mathcal{X} : |f(x)| \le 1,\, \text{for all}\, x \in A \}.$ Then define the \emph{bipolar} by 
    \[
    A^{\circ\circ} = \{ x \in X : |f(x)| \le 1,\,\text{for all}\, f \in A^\circ \}.
    \]
    Also, let $\sigma(X,\mathcal{X})$ denote the weak topology on $X$ generated by $\mathcal{X}$.

    \begin{theorem}\cite[1.5, Theorem]{SchaeferWolff1999}
    Let $A \subset X$. Then $A^{\circ\circ} = \overline{\operatorname{co}}^{\sigma(X,\mathcal{X})}(A \cup \{\theta\}).$
    \end{theorem}

	With these preliminary results in hand, we now present the main results of this article.
	
	\section{Results}
	\medskip
	
	We first obtain a characterization of relative spear operator in the same spirit as that of Ardalani obtained a characterization for spear operators. In other words, we provide a necessary and sufficient conditions for $n^{(2)}_{G}(X,Y)=1$.

	\begin{theorem}
		Let $G\in L(X, Y )$ with $\|G\|=1$. Then for $T\in L(X, Y )$, the following two conditions are equivalent:
		\begin{itemize}
			\item[(i)] $\max_{\omega\in S^1}\|G+\omega T\|_G = 1+\|T\|_G.$
			\item[(ii)] $\nu_G(T)=\|T\|_G$, i.e., $n^{(2)}_{G}(X,Y)=1$.
		\end{itemize}
	\end{theorem}
	
	\begin{proof}
		$(i)\implies (ii).$ Since $V_G(T)\subset S_G(T)$, it follows that $\nu_G(T)\le \|T\|_G.$ Choose $\omega\in S^1$ such that $\|G+\omega T\|_G=1+\|T\|_G$. By Proposition \ref{P2.2}, there exist $x_n\in S_X$ and $y_n^*\in S_{Y^*}$ such that
		\[
		\|Gx_n\|\to 1\qquad\text{and}\qquad y_n^*((G+\omega T)x_n)\to t,
		\]
		where $|t|=1+\|T\|_G$. Set $a_n:=y_n^*(Gx_n)$ and $b_n:=y_n^*(Tx_n).$ Then $a_n+\omega b_n\to t$. Since $|a_n|\le \|Gx_n\|\to1$, we have
		$\limsup |a_n|\le 1$. Moreover, by the definition of $\|\cdot\|_G$,
		$\limsup |b_n|\le \|T\|_G$. It follows that 
		\[
		1+\|T\|_G = |t| \le \limsup(|a_n|+|b_n|) \le \limsup|a_n|+\limsup|b_n| \le 1+\|T\|_G.
		\]
		Hence $\limsup(|a_n|+|b_n|)=1+\|T\|_G$, which implies $\limsup|a_n|=1$ and $\limsup|b_n|=\|T\|_G$. Passing to a subsequence $(n_k)$ we may assume $|a_{n_k}|\to 1$ and $|b_{n_k}|\to\|T\|_G.$ Define $z_k^*=e^{-i\arg(a_{n_k})}y_{n_k}^*$. Then $\|z_k^*\|=1$ and
		\[
		\Re z_k^*(Gx_{n_k})=|a_{n_k}|\to1, \qquad |z_k^*(Tx_{n_k})|=|b_{n_k}|\to\|T\|_G.
		\]
		By Remark \ref{R2.4}, any limit point $t_0$ of $z_k^*(Tx_{n_k})$ belongs to $V_G(T)$. Since $|z_k^*(Tx_{n_k})|\to\|T\|_G$, we obtain $|t_0|=\|T\|_G$, and therefore $\nu_G(T)\ge \|T\|_G$. Combining this with $\nu_G(T)\le \|T\|_G$ gives $\nu_G(T)=\|T\|_G.$ This completes the proof.

		\medskip

		$(ii)\implies (i).$ For any $\omega\in S^1$,
		\[
		\|G + \omega T\|_G \le \|G\|_G + \|\omega T\|_G \le \|G\| + \|\omega T\|_G = 1 + \|T\|_G.
		\]
		Therefore, $\max_{\omega \in S^1} \|G + \omega T\|_G \leq 1 + \|T\|_G.$ We next prove the reverse inequality. Suppose $\nu_G(T) = \|T\|_G$. Then there exists $\lambda \in V_G(T)$ with $|\lambda| = \|T\|_G$. By the definition of $V_G(T)$, taking $\delta_n=1/n$ we obtain sequences $x_n\in S_X$ and
		$y_n^*\in S_{Y^*}$ satisfying
		\[
		\Re y_n^*(Gx_n)>1-\tfrac1n,\qquad |y_n^*(Tx_n)-\lambda|<\tfrac1n.
		\]
		In other words, $\Re y_n^*(Gx_n) \to 1$ and $y_n^*(Tx_n) \to \lambda.$ Since $|y_n^*(Gx_n)|\le \|Gx_n\|\le 1$, the relations $|z_n|\le1$ and $\Re z_n\to 1$ imply $z_n\to1$. Thus, $y_n^*(Gx_n)\to 1$ and $\|Gx_n\|\to 1$. Take $\omega_0\in S^1$ such that $|\lambda|=\omega_0\lambda$. Then
		\[
		y_n^*((G + \omega_0 T)x_n) = y_n^*(Gx_n) + \omega_0 y_n^*(Tx_n) \to 1 + \omega_0\lambda = 1 + |\lambda| = 1 + \|T\|_G.
		\]
		Fix $\delta > 0$. Given any $\varepsilon>0$, since $\|Gx_n\| \to 1$, there exists $N$ such that $\|Gx_n\| > 1 - \delta$ and $|y_n^*((G + \omega_0 T)x_n) - (1 + \|T\|_G)| < \varepsilon$, for all $n \ge N$. Thus,  
		\[
		1 + \|T\|_G \in \overline{\{ y^*((G + \omega_0 T)x) : y^* \in S_{Y^*}, x \in S_X, \|Gx\| > 1 - \delta \}}.
		\]
		As this holds for every $\delta > 0$, we obtain
		\[
		1 + \|T\|_G \in \bigcap_{\delta > 0} \overline{\{ y^*((G + \omega_0 T)x) : y^* \in S_{Y^*}, x \in S_X, \|Gx\| > 1 - \delta \}} = S_G(G + \omega_0 T).
		\]
		Therefore, $\max_{\omega \in S^1} \|G + \omega T\|_G\geq \|G + \omega_0 T\|_G \ge 1 + \|T\|_G$. Combining the two inequalities yields $\max_{\omega \in S^1} \|G + \omega T\|_G = 1 + \|T\|_G$. This completes the proof.
	\end{proof}

    It is easy to see from the definition of relative spear operator that every spear operator is a relative spear operator. In the following example, we have shown that there exists a relative spear operator which is not a spear operator.

    \begin{example}
        Let $X = Y = \ell_1^2$ (over $\mathbb{R}$ or $\mathbb{C}$). Then $Y^* = \ell_\infty^2$. Define $G \in L(X,Y)$ by $G(x_1,x_2) = (x_1,0).$ Clearly $\|G\|=1$. Let $T \in L(X,Y)$ be arbitrary, and write
        \[
        T = \begin{pmatrix} 
        a & b \\ c & d 
        \end{pmatrix}, \quad \text{i.e.,}\quad T(x_1,x_2) = (a x_1 + b x_2, c x_1 + d x_2).
        \]
        We first compute $\|T\|_G$. Let $x:=(x_1, x_2)$. The condition $\|Gx\| = |x_1| = 1$, together with $|x_1|+|x_2|=1$, implies $x_2=0$ and $|x_1|=1$. Thus, $x=(x_1,0)$ with $|x_1|=1$. Then
        \[
        Tx = (ax_1,\; cx_1), \quad \|Tx\|_1 = |a| + |c|.
        \]
        Therefore, $\|T\|_G = |a| + |c|.$ We next compute $\nu_G(T)$. Let $x\in S_X$ and $y^*:=(y_1,y_2)\in S_{Y^*}$. The condition $y^*(Gx)=1$ becomes $y_1 x_1 = 1.$ Since $|y_1|\le1$ and $|x_1|\le1$, this forces $|x_1|=|y_1|=1$ and $y_1=\overline{x_1}$. Therefore, $x_2=0$ and $x=(x_1,0)$ with $|x_1|=1$. Then $Tx = (a x_1,\; c x_1),$ and $y^*(Tx) = y_1 a x_1 + y_2 c x_1 = a + y_2 c x_1,$ where $|y_2|\le1$. Thus, 
        \[
        \nu_G(T) = \max\{|a + y_2 c x_1|: |x_1|=1,\ |y_2|\le1\}= |a| + |c|.
        \]
        Therefore, $\nu_G(T) = \|T\|_G$ for all $T \in L(X,Y).$ Hence $G$ is a relative spear operator. Consider $T \in L(X,Y)$ defined by $T(x_1,x_2) = (0,x_2).$ Then $\|T\|=1$. For any $\lambda \in S^1$, $(G+\lambda T)(x_1,x_2) = (x_1,\ \lambda x_2),$ and therefore
        \[
        \|(G+\lambda T)(x_1,x_2)\|_1 = |x_1| + |x_2| = \|(x_1,x_2)\|_1.
        \]
        It follows that $\|G+\lambda T\|=1 < 2 = 1+\|T\|$, so $G$ is not a spear operator.
        \end{example}

	We next show that this characterization is preserved under composition of surjective isometries.

	\begin{theorem}
		Let $G \in  S_{L(X, Y)}$ be a relative spear operator, and let $U \in L(Y,Z)$ be a surjective isometry. Then $U G$ is a relative spear operator.
	\end{theorem}
	\begin{proof}
		Since $U$ is a surjective isometry, it is an isometric isomorphism with inverse $U^{-1}: Z \to Y$ that is also an isometry. For any $T \in L(X, Z)$, define $M = U^{-1} T \in  L(X, Y)$, so that $T = U M$. We first show that $\|T\|_{U G} = \|M\|_G$. For any $\delta > 0$, note that 
		\begin{align*}
			\{\|Tx\|: x \in S_X, \|U G x\| > 1 - \delta\} & = \{\|Tx\|: x \in S_X, \|G x\| > 1 - \delta\}\\
			& =\{\|U Mx\|: x \in S_X, \|G x\| > 1 - \delta\}\\
			& =\{\|Mx\|: x \in S_X, \|G x\| > 1 - \delta\}.
		\end{align*}
		Taking supremum over these sets and then the limit as $\delta \to 0$ yields $\|T\|_{U G} = \|M\|_G$. Now, for any $\omega \in S^1$, we have
		\[
		\|(U G + \omega T) x\| = \|U (G x + \omega M x)\| = \|G x + \omega M x\|
		\]
		for all $x \in X$. Consequently, $\|U G + \omega T\|_{U G} = \|G + \omega M\|_G.$ Since $G$ is a relative spear operator, $\max_{\omega\in S^1}\|G + \omega M\|_G = 1 + \|M\|_G.$ Thus,
		\[
		\max_{\omega \in S^1} \|U G + \omega T\|_{U G} = \max_{\omega \in S^1} \|G + \omega M\|_G = 1 + \|M\|_G = 1 + \|T\|_{U G}.
		\]
		This shows that $U G$ is a relative spear operator.
	\end{proof}
	
	We now state a simple observation describing the relationship among the numerical indices $n_G(X, Y)$, $n^{(1)}_G(X, Y)$, and $n^{(2)}_G(X, Y)$.
	
	\begin{lemma}\label{L3}
		Let $G\in S_{ L(X, Y )}$. Then $n_G^{(1)}(X,Y)\, n_G^{(2)}(X,Y)\le n_G(X,Y).$
	\end{lemma}
	
	\begin{proof}
		By the definition of $n_G^{(1)}(X,Y)$ and $n_G^{(2)}(X,Y)$, for every
		$T\in L(X, Y )$,
		\[
		n_G^{(1)}(X,Y)\|T\|\le \|T\|_G
		\qquad\text{and}\qquad
		n_G^{(2)}(X,Y)\|T\|_G\le \nu_G(T).
		\]
		Combining these inequalities yields $n_G^{(1)}(X,Y) n_G^{(2)}(X,Y) \|T\|\le \nu_G(T)$ for all $T\in L(X, Y ).$ Since $n_G(X,Y)$ is defined as the largest constant $k\ge 0$ such that $k\|T\|\le \nu_G(T)$ for all $T\in L(X, Y )$, it follows that $n_G^{(1)}(X,Y)\, n_G^{(2)}(X,Y)\le n_G(X,Y),$ which completes the proof.
	\end{proof}
			
			
			
			
		

	On the other hand, for $G\in S_{ L(X, Y )}$ we have $n_G^{(1)}(X,Y)\ge n_G(X,Y).$ Therefore, by combining this with Lemma \ref{L3} we obtain the following consequence.
	
	\begin{proposition}
		Let $X$ and $Y$ be Banach spaces and let $G\in S_{ L(X, Y )}$. If $n_G^{(2)}(X,Y)=1$ then $n_G^{(1)}(X,Y)= n_G(X,Y).$
	\end{proposition}
	
	Finally, we verify that the numerical index, $n_G^{(2)}(X, Y)$, is invariant under isometric isomorphisms.
	
	\begin{theorem}
		Let $X'$ and $Y'$ be Banach spaces and let $U_1: X \to X'$ and $U_2: Y \to Y'$ be isometric isomorphisms. Then $n_G^{(2)}(X,Y) = n_{U_2GU_1^{-1}}^{(2)}(X',Y').$
	\end{theorem}
	
	\begin{proof}
		Set $G' = U_2GU_1^{-1} \in L(X',Y')$. For any operator $T \in L(X,Y)$ define $T' = U_2TU_1^{-1} \in L(X',Y')$. Since $U_1$ and $U_2$ are isometric isomorphisms, the maps
		\[
		x \mapsto x' = U_1x \quad\text{and}\quad y^* \mapsto y^{*'} = y^* \circ U_2^{-1}
		\]
		are bijections from $S_X$ onto $S_{X'}$ and from $S_{Y^*}$ onto $S_{(Y')^*}$, respectively, preserving norms. Moreover,
		\[
		\norm{G'x'} = \norm{U_2GU_1^{-1}U_1x} = \norm{U_2Gx} = \norm{Gx},
		\]
		\[
		\norm{T'x'} = \norm{U_2TU_1^{-1}U_1x} = \norm{U_2Tx} = \norm{Tx},
		\]
		and for any $x \in S_X,\ y^* \in S_{Y^*}$,
		\[
		y^{*'}(G'x') = y^*(U_2^{-1}U_2GU_1^{-1}U_1x) = y^*(Gx),\quad
		y^{*'}(T'x') = y^*(U_2^{-1}U_2TU_1^{-1}U_1x) = y^*(Tx).
		\]
		For a fixed $\delta > 0$ we have
		\[
		\sup\{\norm{Tx}: x\in S_X, \norm{Gx}>1-\delta\}= \sup\{\norm{T'x'}: x'\in S_{X'}, \norm{G'x'}>1-\delta\},
		\]
		because the condition $\norm{Gx}>1-\delta$ is equivalent to $\norm{G'x'}>1-\delta$ under the bijection $x' = U_1x$, and $\norm{Tx} = \norm{T'x'}$. Taking the infimum over $\delta>0$ yields $\norm{T}_G = \norm{T'}_{G'}$. Similarly, for a fixed $\delta > 0$,
		\[
		\sup\Bigl\{ |y^*(Tx)| : \substack{x\in S_X, y^*\in S_{Y^*}, \\ \Re y^*(Gx)>1-\delta} \Bigr\} = \sup\Bigl\{ |y^{*'}(T'x')| : \substack{x'\in S_{X'}, y^{*'}\in S_{(Y')^*}, \\ \Re y^{*'}(G'x')>1-\delta} \Bigr\},
		\]
		since the bijections $x \leftrightarrow x'$ and $y^* \leftrightarrow y^{*'}$ preserve the condition $\Re y^*(Gx)>1-\delta$ and the value $|y^*(Tx)| = |y^{*'}(T'x')|$. Taking the infimum over $\delta>0$ gives $\nu_G(T) = \nu_{G'}(T')$. The correspondence $T \leftrightarrow T'$ is a bijection between operators in $L(X, Y)$ with $\norm{T}_G = 1$ and operators in $L(X', Y')$ with $\norm{T'}_{G'} = 1$. Therefore
		\[
		n_G^{(2)}(X,Y) = \inf_{\norm{T}_G=1} \nu_G(T) = \inf_{\norm{T'}_{G'}=1} \nu_{G'}(T') = n_{G'}^{(2)}(X', Y').
		\]
		This completes the proof.
	\end{proof}

	\subsection{Dual structure and smoothness of the $G$-norm}
	
	We now analyze the geometric structure of $(L(X,Y),\|\cdot\|_G)$, focusing on its dual space and smoothness properties. We denote the norm of $( L(X, Y), \|\cdot\|_G)^*$ as $\|.\|_{G, *}$. Before going to the result, we present the following definition. For $x \in X$ and $y^* \in Y^*$, define the functional $\psi_{x,y^*}:=y^* \otimes x \in L(X,Y)^*$ by
    \[
    \psi_{x,y^*}(S)=[y^* \otimes x](S) = y^*(Sx), \qquad S \in L(X,Y).
    \]
    For more study on smoothness we refer to the reader articles \cite{Kadets1, SR, RoySain, Sain} and the references therein.

    \begin{lemma}\label{dual:lem1}
        Let $G \in L(X,Y)$ with $\|G\|=1$, and $T \in L(X,Y)$. Define
        \[
        C :=\Bigl\{ \psi \in B_{L(X,Y)^*} : \exists\ \text{a net}\ (x_\alpha, y_\alpha^*) \subset S_X \times S_{Y^*},\ \|Gx_\alpha\|\to 1,\ \psi_{x_\alpha, y_\alpha^*} \xrightarrow{w^*} \psi\Bigr\}.
        \]
        Then
        \[
        \sup \Bigl\{ |\lambda| : \exists\ \text{a net}\ (x_\alpha, y_\alpha^*) \subset S_X \times S_{Y^*},\ \|Gx_\alpha\| \to 1,\ y_\alpha^*(Tx_\alpha)\to\lambda \Bigr\} = \sup_{\psi \in C} |\psi(T)|.
        \]
    \end{lemma}

    \begin{proof}
        Let $\psi \in C$. Then, by definition of $C$, there exists a net $(x_\alpha, y_\alpha^*) \subset S_X \times S_{Y^*}$ such that $\|Gx_\alpha\| \to 1$ and $y_\alpha^* \otimes x_\alpha \xrightarrow{w^*} \psi.$ Consequently, we obtain $\lim_{\alpha} y_\alpha^*(Tx_\alpha)= \psi(T).$ Therefore,
        \[
        |\psi(T)| \le |\psi(T)| \le \sup \Bigl\{ |\lambda| : \exists\ \text{a net}\ (x_\alpha, y_\alpha^*) \subset S_X \times S_{Y^*},\ \|Gx_\alpha\| \to 1, y_\alpha^*(Tx_\alpha)\to\lambda \Bigr\}.
        \]
        Since $\psi \in C$ was arbitrary, we have
        \[
        \sup_{\psi \in C} |\psi(T)|\le |\psi(T)| \le \sup \Bigl\{ |\lambda| : \exists\ \text{a net}\ (x_\alpha, y_\alpha^*) \subset S_X \times S_{Y^*},\ \|Gx_\alpha\| \to 1, y_\alpha^*(Tx_\alpha)\to\lambda \Bigr\}.
        \]
        Conversely, let $\lambda$ be such that there exists a net $(x_\alpha, y_\alpha^*) \subset S_X \times S_{Y^*}$ with $\|Gx_\alpha\| \to 1$ and $y_\alpha^*(Tx_\alpha)\to\lambda.$ Since $(\psi_{x_\alpha, y_\alpha^*})$ lies in the weak*-compact set $B_{L(X,Y)^*}$, we may pass to a further subnet $(\alpha_\beta)$ such that $\psi_{x_{\alpha_\beta}, y_{\alpha_\beta}^*} \xrightarrow{w^*} \psi$, for some $\psi \in B_{L(X,Y)^*}$. By construction, $\psi \in C$. Now,
        \[
        \lambda = \lim_{\beta} y_{\alpha_\beta}^*(Tx_{\alpha_\beta}) = \lim_{\beta} [\psi_{x_{\alpha_\beta}, y_{\alpha_\beta}^*}](T) = \psi(T).
        \]
        Therefore, $|\lambda| = |\psi(T)| \le \sup_{\varphi \in C} |\varphi(T)|.$ Taking supremum over all such $\lambda$, we obtain
        \[
        \sup \Bigl\{ |\lambda| : \exists\ \text{a net}\ (x_\alpha, y_\alpha^*) \subset S_X \times S_{Y^*},\ \|Gx_\alpha\| \to 1, y_\alpha^*(Tx_\alpha)\to\lambda \Bigr\}\le \sup_{\psi \in C} |\psi(T)|.
        \]
        Thus, we have the required equality. This completes the proof.
    \end{proof}

    Continuing with the same notations, by using Lemma \ref{dual:lem1} we obtain a description of $B_{( L(X, Y), \|\cdot\|_G)^*}$ in the following result. For more study regarding dual unit ball one can follow the articles \cite{Godefroy, Spear, Kadets2, Generatingoperator} and the references therein.
    
	\begin{theorem}\label{Th1}
		Let $G\in L(X, Y)$ with $\|G\| = 1$. Then $B_{( L(X, Y), \|\cdot\|_G)^*} = \overline{\operatorname{co}}^{\,w^*}(C).$
	\end{theorem}

	\begin{proof}
		We first prove that $C\subseteq B_{(L(X, Y),\|\cdot\|_G)^*}$. Let $\psi\in C$. Then there exists a net $((x_i,y_i^*))\subset S_X\times S_{Y^*}$ such that $\|Gx_i\|\to1$ and $\psi_{x_i,y_i^*}\xrightarrow{w^*}\psi$. Fix $T\in L(X, Y )$ and $\delta>0$. Since $\|Gx_i\|\to1$, there exists $i_0$ such that $\|Gx_i\|>1-\delta$ for all $i\ge i_0$. Hence for all $i\ge i_0$,
		\[
		|\psi_{x_i,y_i^*}(T)|=|y_i^*(Tx_i)|\le\sup\bigl\{|z^*(Tw)|: w\in S_X, z^*\in S_{Y^*}, \|Gw\|>1-\delta\bigr\}.
		\]
		Passing to the limit and then taking the infimum over $\delta>0$, we obtain $|\psi(T)|\le \|T\|_G$. As $T$ is taken arbitrary, thus,  $\|\psi\|_{G,*}\le1$, and therefore $\psi\in B_{(L(X, Y),\|\cdot\|_G)^*}$. We next prove that $\sup_{\psi\in C}|\psi(T)|=\|T\|_G$ for all $T\in L(X, Y)$. Let $T\in (L(X, Y),\|\cdot\|_G)$. From Proposition \ref{DP1} and \ref{P2.2} observe that 
        \begin{align*}
        \|T\|_G=&\sup \Bigl\{ |\lambda| :\exists\ (x_n,y_n^*) \subset S_X \times S_{Y^*},\ \|Gx_n\|\to 1,\ y_n^*(Tx_n)\to \lambda \Bigr\} \\
        = &\sup \Bigl\{ |\lambda| :\exists\ (x_n,y_n^*) \subset S_X \times S_{Y^*},\ \|Gx_n\|\to 1,\ [y_n^* \otimes x_n](T)\to \lambda \Bigr\} \\
        = &\sup \Bigl\{ |\psi(T)| : \exists\ \text{a net}\ (x_{n_\alpha}, y_{n_\alpha}^*) \subset S_X \times S_{Y^*},\ \|Gx_{n_\alpha}\|\to 1,\ y_{n_\alpha}^* \otimes x_{n_\alpha} \xrightarrow{w^*} \psi\Bigr\} \\
        = &\sup_{\psi \in C} |\psi(T)|.
        \end{align*}
        It follows that $B = \{ T \in  L(X, Y) : |\psi(T)| \leq 1,\,\text{for all}\, \psi \in C \} = C^\circ,$ the polar of $C$. Now, by using the Bipolar theorem we have,
		\[
		C^{\circ \circ} = \overline{\operatorname{co}}^{\,w^*}(C \cup \{\theta\}).
		\]  
		But $ C^{\circ \circ} = (C^\circ)^\circ = B^\circ = B_{(L(X, Y),\|\cdot\|_G)^*} $. Moreover, $ C $ is symmetric: if $ \psi \in C $ via a net $ (x_i, y_i^*) $, then $ -\psi \in C $ via the net $ (x_i, -y_i^*) $. Therefore, $\theta$ lies in the convex hull of $ C $, and we obtain  
		\[
		B_{(L(X, Y),\|\cdot\|_G)^*} = \overline{\operatorname{co}}^{\,w^*}(C).
		\]
		This completes the proof.
	\end{proof}
	
	In the finite-dimensional case, the dual description simplifies as follows.
	
	\begin{theorem}\label{Th2}
		Let $X$ and $Y$ be finite-dimensional Banach spaces, let $G \in  L(X, Y)$ with $\|G\| = 1$. Define $\widetilde{C}: = \bigl\{ \psi_{x,y^*}\in B_{L(X,Y)^*} : (x,y^*) \in S_X \times S_{Y^*},\ \|Gx\| = 1 \bigr\}$. Then $B_{( L(X, Y), \|\cdot\|_G)^*} = \operatorname{co}(\widetilde{C}).$
	\end{theorem}
	
	\begin{proof}
		It is not difficult to see from Proposition \ref{DP1} and \ref{P2.5} that $\widetilde{C}\subseteq B_{(L(X, Y),\|\cdot\|_G)^*}$. Since $S_X$ and $S_{Y^*}$ are compact, the condition $ \|Gx\| = 1$ defines a closed subset of $ S_X$ and consequently, the set of pairs $ (x,y^*) $ satisfying it, is compact. Thus, $\widetilde{C}$ is a compact subset of the finite-dimensional Banach space $ L(X, Y)^*$. For any $ T \in  L(X, Y) $, the definition of $ \|T\|_G $ and the compactness imply that the supremum is attained:
		\[
		\|T\|_G = \max \{|y^*(Tx)|: (x,y^*) \in S_X \times S_{Y^*}, \|Gx\| = 1\}= \max_{\psi \in \widetilde{C}} |\psi(T)|.
		\]
		Thus, $B_{( L(X, Y), \|\cdot\|_G)}=\{T\in L(X, Y):|\psi(T)| \le 1,\, \text{for all}\, \psi \in \widetilde{C}\}= \widetilde{C}^\circ.$ Applying the Bipolar theorem we obtain $B^\circ = \overline{\operatorname{conv}}(\widetilde{C}).$ But $B^\circ = (\widetilde{C}^\circ)^\circ = B_{( L(X, Y), \|\cdot\|_G)^*}$. Since $\widetilde{C}$ is compact, $B_{( L(X, Y), \|\cdot\|_G)^*} = \operatorname{conv}(\widetilde{C})$.
	\end{proof}
	
	We next investigate the smooth points of the unit ball of $( L(X, Y ),\|\cdot\|_G)$. The next lemma provides a sufficient condition of support functional at some element in $( L(X, Y ),\|\cdot\|_G)$. By using this lemma we characterizes the said smooth points. Before going to the results, we denote
	\[
	\mathcal{A}_G(T):=\bigl\{(x_\alpha,y_\alpha^*)\in S_X\times S_{Y^*}:\|Gx_\alpha\|\to1,\;y_\alpha^*(Tx_\alpha)\to 1\bigr\}.
	\]
	
	\begin{lemma}\label{lem:A}
		Let $G\in L(X, Y)$ with $\|G\|=1$. Let $T\in L(X, Y)$ be such that $\|T\|_G=1$. If $\varphi\in \overline{\{\psi_{x_\alpha,y_\alpha^*}:(x_\alpha,y_\alpha^*)\in\mathcal A_G(T)\}}^{w^*}$ then $\varphi(T)=1$ with $\|\varphi\|_{G, *}=1$ i.e., $\varphi$ is a support functional of $T$ in $(L(X, Y), \|.\|_G)^*$.
	\end{lemma}
	\begin{proof}
		Fix $\varepsilon>0$ and $S\in  L(X, Y)$. By the definition of $\|\cdot\|_G$, there exists $\delta>0$ such that
		\[
		\sup\bigl\{|y^*(Sx)|:x\in S_X, y^*\in S_{Y^*}, \|Gx\|>1-\delta\bigr\}
		<\|S\|_G+\varepsilon.
		\]
		Let $(x_\alpha,y_\alpha^*)\in\mathcal A_G(T)$. Since $\|Gx_\alpha\|\to 1$, we can choose $\alpha_0$ so that $\|Gx_\alpha\|>1-\delta$ for all $\alpha\geq\alpha_0$.
		Consequently, for every $\alpha\geq\alpha_0$, $|y_\alpha^*(Sx_\alpha)|<\|S\|_G+\varepsilon.$ By the given hypothesis, $\varphi(S)=\lim_\alpha y_\alpha^*(Sx_\alpha)\le\|S\|_G+\varepsilon.$ Since $\varepsilon>0$ was arbitrary, we obtain $\varphi(S)\le\|S\|_G$ for all $S\in  L(X, Y)$. Therefore, $\|\varphi\|_{G, *}\le 1$. Moreover, $\varphi(T)=\lim_\alpha y_\alpha^*(Tx_\alpha)=1=\|T\|_G.$ Consequently, $\|\varphi\|_{G, *}=1$. This completes the proof.
	\end{proof}

	This description allows us to characterize smooth points of the unit ball in $(L(X, Y), \|.\|_G)$.

	\begin{theorem}\label{thm:smooth-G-infinite}
		Let $X,Y$ be Banach spaces and let $G\in L(X, Y )$ with $\|G\|=1$.
		Assume that $\|\cdot\|_G$ is a norm on $ L(X, Y )$.
		Let $T\in L(X, Y )$ satisfy $\|T\|_G=1$.
		The following assertions are equivalent:
		\begin{enumerate}
			\item[(i)]
			$T$ is a smooth point of the unit ball of $( L(X, Y ),\|\cdot\|_G)$.
			\item[(ii)]
			$\overline{\{\psi_{x_\alpha,y_\alpha^*}:(x_\alpha,y_\alpha^*)\in\mathcal A_G(T)\}}^{w^*}=\{\varphi\}$, where $\varphi (T)=1$. 
		\end{enumerate}
	\end{theorem}
	\begin{proof}
		$\textit{(i)}\implies \textit{(ii)}.$
		Assume that $T$ is smooth in $(L(X, Y), \|\cdot\|_G)$. Then there exists a unique support functional $\varphi$ in
		$S_{( L(X, Y ), \|.\|_G)^*}$ such that $\varphi(T)=1$. Thus, it immediately follows from Lemma \ref{lem:A} that $\overline{\{\psi_{x_\alpha,y_\alpha^*}:(x_\alpha,y_\alpha^*)\in\mathcal A_G(T)\}}^{w^*}=\{\varphi\}$, where $\varphi (T)=1$.
		
		\medskip
		
		\noindent
		$\textit{(ii)}\implies \textit{(i)}.$
		It is clear from Lemma \ref{lem:A} that $\varphi\in J_{\|.\|_G}(T)$. To prove $T$ is smooth in $(L(X, Y), \|.\|_G)$, it is enough to prove that $J_{\|.\|_G}(T)=\{\varphi\}$. Observe that $J_{\|.\|_G}(T)$ is a non-empty, weak*-compact, convex subset of $B_{(L(X, Y), \|.\|_G)^*}$. By the Krein-Milman theorem, $J_{\|.\|_G}(T)$ is the weak$^*$ closed convex hull of its extreme points. Let $\psi_0$ be an extreme point of $J_{\|.\|_G}(T)$. Because $J_{\|.\|_G}(T)$ is a face of $B_{( L(X, Y), \|.\|_G)^*}$, $\psi_0$ is also an extreme point of $B_{( L(X, Y), \|.\|_G)^*}$. We know from Theorem \ref{Th1} that $B_{( L(X, Y), \|\cdot\|_G)^*} = \overline{\operatorname{co}}^{\,w^*}(C)$. Also, by Milman's converse, every extreme point of $B_{( L(X, Y), \|\cdot\|_G)^*}$ lies in the weak$^*$ closure of $C$. Therefore, there exists a net $(x_i,y_i^*)\subset S_X\times S_{Y^*}$ with $\|Gx_i\|\to 1$ and $\psi_{x_i,y_i^*}\xrightarrow{w^*}\psi_0$. Since $\psi_0(T)=1$, therefore,
		\[
		\lim_i y_i^*(Tx_i)=\lim_i\psi_{x_i,y_i^*}(T)=\psi_0(T)=1.
		\]
		Consequently, $(x_i,y_i^*)\in \mathcal{A}_G(T)$, so $\psi_0\in \overline{\{\psi_{x_\alpha,y_\alpha^*}:(x_\alpha,y_\alpha^*)\in\mathcal A_G(T)\}}^{w^*}$. However, 
		\[
		\overline{\{\psi_{x_\alpha,y_\alpha^*}:(x_\alpha,y_\alpha^*)\in\mathcal A_G(T)\}}^{w^*}=\{\varphi\}.
		\]
		Consequently, $\psi_0=\varphi$. Thus, every extreme point of $J_{\|.\|_G}(T)$ equals $\varphi$. Since $J_{\|.\|_G}(T)$ is the weak* closed convex hull of its extreme points, we obtain $J_{\|.\|_G}(T)=\{\varphi\}$, as desired. Thus, the proof is complete.
	\end{proof}

	\subsection{Hilbert space applications and further properties}
	
	As an application to our study, we now specialize to Hilbert spaces. We first consider the finite-dimensional Hilbert space case. In this case, whenever we use the definitions of $\nu_G(T)$ and $\|T\|_G$, we use Proposition \ref{P2.5} as our reference.
    
	
	\begin{theorem}\label{Th3.11}
		Let $H_{1}$ and $H_{2}$ be finite-dimensional Hilbert spaces, and let $G \in S_{L(H_{1}, H_{2})}$ be a relative spear operator. Then $G$ is a partial isometry.
	\end{theorem}
	
	\begin{proof}
		By a singular value decomposition of $G$, for every $x\in H_1$ we define
		\[
		Gx=\sum_{j=1}^{n} s_j \langle x,u_j\rangle v_j,
		\]
		where $s_1\ge s_2\ge\dots\ge s_n\ge0$ are the singular values of $G$, $\{u_j\}_{j=1}^n$ is an orthonormal basis of $H_1$ and $\{v_j\}_{j=1}^n$ is an orthonormal set in $H_2$. Since $\|G\|=1$, we have $s_1=1$. Suppose, for a contradiction, that $G$ is not a partial isometry. Then there exists an index $i$ such that $0<s_{i}<1$. Let $i$ be the smallest such index, so that $s_{1}=\dots =s_{i-1}=1$ and $s_{i}<1$. Define the rank-one operator $T=v_{i}\otimes u_{1}$, i.e., $Tx=\langle x,u_{1}\rangle v_{i}$. We show that $\|T\|_G\neq \nu_{G}(T)$. First we show that $\|T\|_{G}=1$. If $x=u_1$, then $\|Gx\|=\|Gu_1\|=s_1=1$. By Choosing $y=v_i$, we obtain $\langle Tx,y\rangle      =\langle v_i,v_i\rangle =1.$ Thus, $1\in S_G(T)$, so $\|T\|_G \ge 1$. Observe that $\|T\|=1$ and we know $\|T\|_{G}\le\|T\|$. Thus, $\|T\|_{G}=1$. Next, we show that $\nu_{G}(T)=0$. Let $\lambda\in V_G(T)$. Then there exist $x\in S_{H_1}$ and $y\in S_{H_2}$ such that $\langle Gx,y\rangle=1$ and $\lambda=\langle Tx, y\rangle.$ It follows that $|\langle Gx, y\rangle|=\|Gx\|\|y\|$ and $\|Gx\|=1$. Consequently, the equality case of the Cauchy-Schwarz inequality implies that $y$ is a scalar multiple of $Gx$. Using $\langle Gx, y\rangle=1$ and $\|Gx\|=\|y\|=1,$ we obtain $y=Gx$. Write
		\[
		x=\sum_{j=1}^n a_j u_j ,
		\qquad \sum_{j=1}^n |a_j|^2=1 .
		\]
		Then $\|Gx\|^2=\sum_{j=1}^n s_j^2 |a_j|^2.$ Since $\|Gx\|=1$ and $s_j\le s_i<1$ for $j\ge i$, it follows that $a_j=0$ for $j\ge i$. Thus, $x\in \operatorname{span}\{u_1,\dots,u_{i-1}\}$ and consequently, $y=Gx\in \operatorname{span}\{v_1,\dots,v_{i-1}\}.$ Now,
		\[
		\lambda=\langle Tx,y\rangle =\langle x,u_1\rangle \langle v_i,y\rangle .
		\]
		Since $y\in\operatorname{span}\{v_1,\dots,v_{i-1}\}$ and $v_i$ is orthogonal to this subspace, we obtain $\langle v_i, y\rangle=0.$ Hence $\lambda=0$. Therefore, $V_G(T)=\{0\}$ and $\nu_G(T)=0$. Since $\|T\|_G=1$ and $\nu_G(T)=0$, we obtain $\|T\|_G\ne \nu_G(T),$ which contradicts that $G$ is a relative spear operator. Hence no singular value of $G$ lies in $(0,1)$. Thus, all non-zero singular values are equal to $1$, and therefore $G^*G$ is an orthogonal projection. Consequently, $G$ is a partial isometry.
	\end{proof}
	
	We next obtain a structural characterization in terms of the norm-attaining set $M_G$. For a subset $E \subset X$, the distance from $x$ to $E$ is defined as $\operatorname{dist}(x,E):=\inf_{y\in E}\|x-y\|.$

	\begin{theorem}\label{Th3.12}
		Let $H$ be a Hilbert space and let 
		$G \in L(H)$ be a norm-attaining operator with $\|G\|=1$. Let $E = \overline{\operatorname{span}}(M_G).$ Then the following are equivalent:
		\begin{enumerate}
			\item[(i)] $\gamma := \sup\{\|Gx\| : x \in S_H \cap E^\perp\} < 1.$
			\item[(ii)] For every $\varepsilon>0$ there exists $\delta>0$ such that
			\[
			x\in S_H,\ \|Gx\|>1-\delta \Longrightarrow \operatorname{dist}(x,E)<\varepsilon.
			\]
			\item[(iii)] For every $T\in L(H)$,
			$\|T\|_G = \sup_{x\in M_G}\|Tx\|.$
		\end{enumerate}
	\end{theorem}
	\begin{proof}
		First, we establish two properties of $G$ on $E$. Since $\|G\|=1$, the operator $I-G^*G$ is positive. For $x\in M_G$ we have $\langle G^*Gx,x\rangle=\|Gx\|^2=1=\langle x, x\rangle.$ Consequently, $\langle (I-G^*G)x,x\rangle=0$. Positivity implies $(I-G^*G)x=0$, so $G^*Gx=x$. By continuity this extends to $E$, and therefore $G^*G=I$ on $E.$ Consequently,
		\begin{enumerate}
			\item[(a)] $G$ is an isometry on $E$ and
			\item[(b)] if $u\in E$ and $v\in E^\perp$, then $\langle Gu,Gv\rangle=0$.
		\end{enumerate}
		Since $E$ is a closed subspace of the Hilbert space $H$, there exists an orthogonal projection $P:H\to E$. For $x\in S_H$ we write $x = Px + (I-P)x,$ where $Px\in E$ and $(I-P)x\in E^\perp$. Let $u=Px$ and $v=(I-P)x$. Then by using the fact in (a) and (b) we have
		\begin{equation}\label{eq1}
			\|Gx\|^2=\|u\|^2+\|Gv\|^2 .
		\end{equation}

		\medskip
		\noindent
		$(i)\implies (ii)$. Assume $\|Gv\|\le \gamma\|v\|$ for $v\in E^\perp$ with $\gamma<1$. Let $\|x\|=1$ then $x=u+v$ for some $u\in E$ and $v\in E^\perp$. Therefore, $u\perp v$ implies $1=\|u\|^2+\|v\|^2$. It then follows from \eqref{eq1} that $\|Gx\|^2\le 1-(1-\gamma^2)\|v\|^2.$ If $\|Gx\|>1-\delta$, then $(1-\delta)^2<1-(1-\gamma^2)\|v\|^2,$ so $\|v\|^2<\frac{2\delta-\delta^2}{1-\gamma^2}.$ It is easy to see that $\operatorname{dist}(x, E)=\|v\|$, thus, by choosing $\delta$ sufficiently small gives $\operatorname{dist}(x,E)<\varepsilon$.

		\medskip
		
		\noindent
		$(ii)\implies (iii).$ Since $M_G \subseteq \{x\in S_H:\|Gx\|>1-\delta\}$, for every $\delta>0$, $\sup_{\|Gx\|>1-\delta}\|Tx\|\ge \sup_{x\in M_G}\|Tx\|.$ Consequently, $\|T\|_G \ge \sup_{x\in M_G}\|Tx\|$. If we fix $\varepsilon>0$, there exists $\delta>0$ such that
		\begin{equation}\label{eq2}
			\|Gx\|>1-\delta \implies \operatorname{dist}(x,E)<\varepsilon.
		\end{equation}
		Choose $z\in E$ with $\|x-z\|<\varepsilon$ and set $y=z/\|z\|\in M_G$. Then $\|x-y\|<2\varepsilon$. Thus, for every $x$ satisfying equation \eqref{eq2} we have 
		\begin{align*}
			\|Tx\|\le \|Ty\| + \|T\|\,\|x-y\|\leq \sup_{y\in M_G}\|Ty\|+2\|T\|\varepsilon,
		\end{align*}
		and consequently, $\|T\|_G \le \sup_{\|Gx\|>1-\delta}\|Tx\| \le \sup_{y\in M_G}\|Ty\|+2\|T\|\varepsilon$. Letting $\varepsilon\to0$ gives $\|T\|_G \le \sup_{x\in M_G}\|Tx\|$. Therefore, $\|T\|_G = \sup_{x\in M_G}\|Tx\|$.

		\medskip

		\noindent
		$(iii)\implies (i).$ $(iii)\implies (i)$. Assume that $(i)$ fails. Then there exist $v_n\in S_H\cap E^\perp$ such that $\|Gv_n\|\to1$. Let $X=\overline{\operatorname{span}}\{v_n\}$ and denote by $P_X$ the orthogonal projection onto $X$. Since $P_Xv_n=v_n$, we have $\|P_Xv_n\|=1$, and because $\|Gv_n\|\to1$ it follows that $1\in S_G(P_X)$. Hence $\|P_X\|_G=1$. Since $X \subset E^\perp$, it follows that $E\subset X^\perp$. Because $M_G \subset E$, we obtain $M_G \subset X^\perp$. Hence $P_X x = 0$ for every $x \in M_G$. Therefore, $\sup_{x\in M_G}\|P_Xx\|=0,$ which contradicts $(iii)$. Thus $(i)$ must hold. 
		
		\medskip
		
		This completes the proof.
	\end{proof}

	We conclude this article with a comparison of $G$-norms. In other words, we obtain a sufficient condition for $\|T\|_{G_1}\le\|T\|_{G_2}$.

	\begin{theorem}
		Let $G_1, G_2 \in L(X)$ with $\|G_1\| = \|G_2\| = 1$. Suppose there exists a function $\phi : (0, 1] \to (0, 1]$ with $\phi(t) \to 1$ as $t \to 1$ such that for every $x \in S_X$,
		\[
		\|G_1x\| \ge t \implies \|G_2x\| \ge \phi(t).
		\]
		Then for every $T \in L(X)$, $\|T\|_{G_1} \le \|T\|_{G_2}$.
	\end{theorem}
	
	\begin{proof}
		Fix $T \in L(X)$. By using the definition of $\|T\|_{G_1}$ via $\widetilde{S}_G(T)$ described in Proposition \ref{P2.2}, let $(x_n) \subset S_X$ be a sequence with $\|G_1x_n\| \to 1$ and $\|Tx_n\| \to \|T\|_{G_1}$. For any $t < 1$, eventually $\|G_1x_n\| \ge t$, so by hypothesis $\|G_2x_n\| \ge \phi(t)$. Since $\phi(t) \to 1$ as $t \to 1$, we conclude that $\|G_2x_n\| \to 1$. Hence $(x_n)$ is also a norming sequence for $G_2$. Therefore,
		\[
		\|T\|_{G_2} \ge \limsup_{n \to \infty} \|Tx_n\| = \|T\|_{G_1}.
		\]
		Thus, the proof is complete.
	\end{proof}
	
	\begin{question}
		What is the infinite-dimensional version of Theorem \ref{Th3.11}?
	\end{question}
	
	\begin{question}
		Can Theorem \ref{Th3.11} and Theorem \ref{Th3.12} be extended to the Banach space setting? 
	\end{question}
	
	\begin{question}
		If the previous answer is negative then find the operators $G$ for which the conclusion of Theorem \ref{Th3.11} continues to hold in the framework of general Banach spaces. 
	\end{question}

	\medskip

	\noindent
	\textbf{Conflict of Interest:} The authors declare no conflicts of interest.

	\bibliographystyle{plain}

    \end{document}